\documentclass[11pt,a4paper,reqno]{amsart}

\usepackage{amsfonts, amssymb, amscd}
\usepackage{amsmath}
\usepackage{verbatim}
\usepackage{amssymb}
\usepackage{mathrsfs}
\usepackage{graphicx}
\usepackage{cite}
\usepackage[all]{xy}
\usepackage{tikz}
\usepackage{tikz-cd}

\usepackage{graphicx}
\usepackage{float}
\usepackage[margin=1.5in]{geometry}
\usepackage[pagebackref]{hyperref}
\newtheorem{theorem}{Theorem}[section]
\newtheorem{lemma}[theorem]{Lemma}
\newtheorem{proposition}[theorem]{Proposition}

\newtheorem{corollary}[theorem]{Corollary}

\newtheorem{conjecture}[theorem]{Conjecture}

\numberwithin{equation}{section}

\DeclareMathOperator{\vol}{vol}
\def\Q{\mathbb{Q}}
\def\C{\mathbb{C}}

\begin{document}

\title{Varieties of general type with doubly exponential asymptotics}

\begin{abstract}
We construct smooth projective varieties of general type with the smallest known volume and others with the most known vanishing plurigenera in high dimensions.  The optimal volume bound is expected to decay doubly exponentially with dimension, and our examples achieve this decay rate. We also consider the analogous questions for other types of varieties. For example, in every dimension
we conjecture the terminal Fano variety of minimal volume,
and the canonical Calabi-Yau variety of minimal volume.
In each case, our examples exhibit doubly exponential behavior.
\end{abstract}

\author{Louis Esser}
\address{UCLA Mathematics Department,
Box 951555, Los Angeles, CA 90095-1555} \email{esserl@math.ucla.edu}

\author{Burt Totaro}
\address{UCLA Mathematics Department,
Box 951555, Los Angeles, CA 90095-1555} \email{totaro@math.ucla.edu}

\author{Chengxi Wang}
\address{UCLA Mathematics Department,
Box 951555, Los Angeles, CA 90095-1555} \email{chwang@math.ucla.edu}

\maketitle

\section{Introduction}\label{section1}
For a smooth complex projective variety $X$ of dimension $n$, the \textit{volume} $\mathrm{vol}(X)$ measures the asymptotic growth of the plurigenera $h^0(X,\ell K_X)$, by
$$\mathrm{vol}(X) := \lim_{\ell \rightarrow \infty} h^0(X,\ell K_X)/(\ell^n/n!).$$
This is equal to the intersection number $K_X^n$ if the canonical class
$K_X$ is ample or (more generally) nef.
A variety is said to be of {\it general type }if its volume is positive. By a theorem of Hacon--M\textsuperscript{c}Kernan, Takayama, and Tsuji, for each positive integer $n$ there is a constant $r_n$ such that the pluricanonical linear system $|\ell K_X|$ gives a birational embedding of $X$
into projective space for every $\ell \geq r_n$ and every smooth projective $n$-fold $X$ of general type \cite{HM06, Takayama, Tsuji}.  This implies a positive lower bound $1/(r_n)^n$ on volume for all $n$-folds of general type.  However, the asymptotics of these bounds for $n$ large remain mysterious.

Following a long tradition in algebraic
geometry \cite{Iano-Fletcher, BPT, Brown-Kasprzyk},
we seek examples of low volume among weighted projective hypersurfaces $X$
with canonical singularities and ample canonical class.  A resolution of singularities of $X$ will then be a smooth projective variety of general type. Weighted projective hypersurfaces
exhibit a huge range of behavior, and finding good examples
is not easy.
In version 1 of \cite{ETTW} on the arXiv, Tao and the authors found examples
of $n$-folds of general type with volume less than $1/e^{n^{3/2}}$,
also showing that the bound $r_n$ must grow at least as $e^{n^{1/2}}$.
In this paper, we considerably improve those results.

\begin{theorem}
\label{main}
\begin{enumerate}
    \item For every integer $n\geq 3$, there is a smooth complex projective variety $X$ of general type with dimension $n$ and volume less than $1/2^{2^{n/2}}$.  It is possible to choose $X$ with geometric genus $p_g = h^0(X,K_X)$ positive.
    \item For every integer $n \geq 2$, there is a smooth complex projective variety $X$ of general type with dimension $n$ such that the linear system $|\ell K_X|$ does not give a birational embedding for any $\ell < 2^{2^{(n-2)/2}}$.
    \item For every integer $n\geq 2$, there is a smooth complex projective variety $X$ of general type with dimension $n$ such that $H^0(X,\ell K_X) = 0$ for $1 \leq \ell < 2^{2^{(n-4)/2}}$.
\end{enumerate}
\end{theorem}

Before the authors' series of papers in 2021, 
the best examples in high dimensions were by Ballico, Pignatelli,
and Tasin. They gave smooth $n$-folds of general type with volume
about $1/n^n$, and with about $n/3$ vanishing plurigenera
\cite[Theorems 1 and 2]{BPT}. Theorem \ref{main} is a big advance.
In particular,
the constants $r_n$ grow at least doubly exponentially with dimension.
Moreover, it is expected that the optimal bound is indeed doubly exponential.
Indeed, in the more general situation of klt pairs of general type with standard coefficients, Koll\'{a}r has proposed a conjecturally optimal example
\cite{Kollarlog}, \cite[Introduction]{HMXbir}:

$$(X,\Delta) = \left( \mathbb{P}^n,\frac{1}{2}H_0+\frac{2}{3}H_1
+\frac{6}{7}H_2+\cdots+\frac{s_{n+1}-1}{s_{n+1}}H_{n+1}\right),$$

where $H_0,\ldots,H_{n+1}$ are $n+2$ general hyperplanes in $\mathbb{P}^n$
and the sequence $s_m$ is Sylvester's sequence, defined recursively by $s_0 = 2$ and $s_m = s_{m-1}(s_{m-1}-1)+1$ for $m \geq 1$.  (To say that the pair $(X,\Delta)$
is of general type with standard coefficients means that $K_X + \Delta$ is big and all coefficients of the $\Q$-divisor $\Delta$ are of the form $1-1/j$
with $j\in {\mathbb{Z}_{+}}$.)

The volume of $K_X+\Delta$ is $1/(s_{n+2}-1)^{n}$, which is, crudely, about
$1/2^{2^n}$. Totaro and Wang constructed klt varieties of general
type (rather than pairs) for which the logarithm of the volume
is asymptotic to the logarithm of the volume
of Koll\'ar's pair \cite{Totaro-Wang}; so those examples
should be close to optimal
for klt varieties.
We now construct {\it smooth }varieties of general type
with volume around $1/2^{2^{n/2}}$. So the minimal volume under these assumptions
should be (crudely)
somewhere between $1/2^{2^{n/2}}$ and $1/2^{2^n}$.

We'll also consider the analogous problems for other classes of varieties.  In total, we'll study 1) canonical Calabi-Yau varieties, 2) terminal Fano varieties, 3) smooth varieties of general type, and 4) terminal Calabi-Yau varieties.
(Our constructions happen to become more complicated in this order.) For each class, we give examples a) of small volume, and b) with many vanishing spaces of sections.  When $X$ is of general type, these are the problems of finding a variety with small (canonical) volume and many vanishing plurigenera, as in Theorem \ref{main}.  For $X$ Fano, we'll consider $\mathrm{vol}(-K_X)$ and many vanishing spaces of sections $H^0(X,-\ell K_X)$, while for $X$ Calabi-Yau, we'll look at $\mathrm{vol}(A)$ and the groups $H^0(X,\ell A)$ with $A$ an ample Weil divisor on $X$.

For problems 1a), 2a), and 1b), we can conjecture {\it optimal }examples,
with supporting evidence in low dimensions. The examples involve
Sylvester's sequence $s_m$. In what follows, a projective variety
is said to be Calabi-Yau if $K_X\sim_{\Q}0$.

\begin{conjecture}[Conjecture \ref{1aconj}]
For a positive integer $n$, let $d = (2s_n - 3)(s_n-1)$. A general hypersurface $X$ of degree $d$ in $\mathbb{P}^{n+1}(d/s_0,\ldots,d/s_{n-1},s_n-1,s_n-2)$
is the canonical Calabi-Yau $n$-fold with ample Weil divisor $\mathcal{O}_X(1)$
of minimal volume.
\end{conjecture}

\begin{conjecture}[Conjecture \ref{2aconj}]
For each integer $n \geq 2$, let $d = (2s_{n-1} - 3)(s_{n-1}-1)$.  Then a general hypersurface $X$ of degree $d$ in
$$\mathbb{P}^{n+1}(d/s_0,\ldots,d/s_{n-2},s_{n-1}-1,s_{n-1}-2,1)$$
is the terminal Fano $n$-fold of minimal anticanonical volume.
\end{conjecture}

For problem 1b), our example is not optimal among all canonical Calabi-Yau varieties, but we expect it to be optimal among weighted projective hypersurfaces.

\begin{conjecture}[Conjecture \ref{1bconj}]
For an integer $n \geq 2$, let $d=(s_{n-1}-1)(3s_{n-1}-4)^2$. Then a general hypersurface $X$ of degree $d$ in
$$\mathbb{P}^{n+1}(d/s_0,\ldots,d/s_{n-2},(s_{n-1}-1)(3s_{n-1}-4),(s_{n-1}-1)(3s_{n-1}-5),3s_{n-1}^2-9s_{n-1} + 7)$$
is the quasi-smooth canonical Calabi-Yau hypersurface of dimension $n$
with ample Weil divisor
$\mathcal{O}_X(1)$ with the largest number $M$ having
$H^0(X,\mathcal{O}_X(\ell))=0$ for $1\leq \ell < M$,
among all such hypersurfaces.
Namely, $M = 3s_{n-1}^2-9s_{n-1}+7 > 2^{2^{n-1}}$.
\end{conjecture}

We produce examples with doubly exponential behavior for all eight problems. The volumes of the conjecturally optimal examples 1a) and 2a) are (crudely) around $1/2^{2^n}$.  Our other examples 3a) and 4a) have volume around
$1/2^{2^{n/2}}$, and the optimal bound should be somewhere between
that and $1/2^{2^n}$.  The number of vanishing spaces of sections is around $2^{2^n}$ in example 1b) and around $2^{2^{n/2}}$ for 2b), 3b), and 4b).

Although our examples have mild singularities in terms of the minimal
model program, note that such exotic behavior cannot occur
for {\it smooth }Fano or Calabi-Yau varieties
(or likewise for smooth projective varieties with ample canonical class,
rather than of general type). This is a qualitative sense in which
our examples are optimal. Indeed, every ample Weil divisor on a smooth
projective variety is Cartier and hence has volume an integer.
Also, the Ambro-Kawamata conjecture predicts that for every klt projective
variety $X$ and every ample {\it Cartier }divisor $A$ with $A-K_X$ ample,
$H^0(X,A)$ is not zero \cite{Ambro,Kawamata,PST}.

Finally, we deduce from our examples that the constant $a_n$
in a Noether-type inequality $\vol(X)\geq a_np_g(X)-b_n$
for smooth projective $n$-folds of general type must be
doubly exponentially small as a function of $n$ (Theorem \ref{noether}).

\noindent{\it Acknowledgements. }Esser and Totaro were
supported by NSF grants DMS-1701237 and DMS-2054553. Thanks to Jungkai
Chen, J\'anos Koll\'ar, and Miles Reid for useful conversations.

\section{Preliminaries on Weighted Projective Spaces}\label{prelim}

Some introductions to the singularities
of the minimal model program, such as terminal, canonical, or
Kawamata log terminal (klt),
are \cite{Reidyoung, KM}.
Throughout, we'll work over the complex numbers, though some statements would remain true in any characteristic.

Given a collection of positive integers $a_0,\ldots,a_N$, the weighted projective space $Y = \mathbb{P}(a_0,\ldots,a_N)$ is the quotient variety
$(\mathbb{A}^{N+1}\setminus 0)/\mathbb{G}_m$, where the multiplicative group $\mathbb{G}_m$ acts by $t(x_0,\ldots,x_N) = (t^{a_0}x_0,\ldots,t^{a_N}x_N)$.  We say that $Y$ is {\it well-formed} when $\gcd(a_0,\ldots,\widehat{a_j},\ldots,a_n) = 1$ for each $j$
\cite[Definition 6.9]{Iano-Fletcher}.
We always assume that $Y$ is well-formed.
(In other words, the analogous
quotient stack $[(A^{n+1}-0)/\mathbb{G}_m]$
has trivial stabilizer group
in codimension 1.) For well-formed $Y$, the canonical divisor of $Y$
is given by $K_Y=\mathcal{O}(-a_0-\cdots-a_N)$.
Here $\mathcal{O}(d)$ is the sheaf
associated to a Weil divisor on $Y$ for every integer $d$. It is a line bundle
if and only if $d$ is a multiple of every weight $a_j$. The volume
of the ample Weil divisor $\mathcal{O}(1)$ is $1/(a_0\cdots a_N)$.

We write $\mathbb{P}(a_0^{(b_0)},\ldots,a_r^{(b_r)})$ for the weighted
projective space with the weight $a_i$ repeated $b_i$ times.

Weighted projective spaces have only cyclic quotient singularities.  To determine whether these are canonical or terminal, we use the Reid-Tai criterion \cite[Theorem 4.11]{Reidyoung}.  For $a$ an integer and $b$ a positive integer,
consider $a \bmod b$ as an integer in the set $\{0,\ldots,b-1\}$.

\begin{theorem} \label{RT}
For a positive integer $r$, let the group $\mu_r$ of $r$th roots of unity act on affine space by $\zeta(t_1,\ldots,t_s) = (\zeta^{b_1}t_1,\ldots,\zeta^{b_s}t_s)$.  The quotient $\mathbb{A}^s/\mu_r$ is said to be a cyclic quotient singularity of type $\frac{1}{r}(b_1,\ldots,b_s)$.  Assume that this description is well-formed in the sense that $\gcd(r,b_1,\ldots,\widehat{b_j},\ldots,b_s) = 1$ for all $j = 1,\ldots,s$.  Then the quotient singularity is canonical (resp.\ terminal) if and only if
$$\sum_{j = 1}^s ib_j \bmod r \geq r $$
(resp.\ $> r$) for all $i = 1,\ldots,r-1$.

\end{theorem}

Each weighted projective space $Y=\mathbb{P}(a_0,\ldots,a_N)$ is a toric variety,
with an action
of the torus $T=(\mathbb{G}_m)^{N+1}/\mathbb{G}_m\cong (\mathbb{G}_m)^N$
by scaling the variables.
Since the locus where $Y$ is canonical or terminal is open
as well as $T$-invariant, we have the following:

\begin{lemma}\label{orbit}
Let $Y$ be a weighted projective space. If $Y$ is canonical (resp.\ terminal) at a point $q$, then $Y$ is also canonical (resp.\ terminal) at all points $p$ such that $q$ is in the closure of the $T$-orbit of $p$.
\end{lemma}

Because of this lemma, to prove that $Y$ is canonical (or terminal), it suffices to check only the coordinate points $[0:\cdots:0:1:0:\cdots:0]$.  Moreover, if $Y$ is canonical at every coordinate point besides $p$, then $Y \setminus p$ is canonical.  Nevertheless, we'll sometimes need to know the singularities of $Y$ away from the coordinate points, in particular on some other stratum $U_I$ of the torus action, where $I \subset \{0,\ldots,N\}$ is the set of indices of nonzero coordinates on this stratum.  Therefore, we state how to identify the quotient singularities at any point of $Y$ (elaborating on the statements in \cite[section 5.15]{Iano-Fletcher}).

\begin{proposition} \label{checksing}
Let $Y = \mathbb{P}(a_0,\ldots,a_N)$ be a weighted projective space and $I \subset \{0,...,N\}$ a nonempty subset with size $|I| = k+1$.  Let $r = \gcd(a_i: i \in I)$.  If $p$ is a point on $U_I$, then $p$ has a neighborhood analytically isomorphic to a quotient singularity of type 
$$\frac{1}{r}(a_i: i \notin I) \times \mathbb{A}^k.$$
\end{proposition}

A closed subvariety $X$ of the weighted projective space $Y$ is {\it quasi-smooth} if its affine cone in $\mathbb{A}^{N+1}$ is smooth away from the origin.  We say that $X$ is {\it well-formed} if $Y$ is well-formed and the codimension of the intersection $X \cap Y_{\mathrm{sing}}$ in $X$ is at least 2.  We'll be considering the case where $X$ is a hypersurface.  In this situation, as long as the degree $d$ is not equal to any of the weights (this assumption will always hold in our examples), every quasi-smooth hypersurface of dimension at least 3 is well-formed \cite[Theorem 6.17]{Iano-Fletcher}.  The following criterion, which works in characteristic zero, determines whether a general hypersurface is quasi-smooth \cite[Theorem 8.1]{Iano-Fletcher}:

\begin{proposition} \label{qsmooth}
A general hypersurface $X$ of degree $d$ in the weighted projective space $\mathbb{P}(a_0,\ldots, a_N)$ is quasi-smooth if and only if one of the following properties holds:
\begin{enumerate}
    \item $a_i = d$ for some $i,$ or
    \item for each nonempty subset $I$ of $\{0,\ldots ,N\}$, either
    \begin{enumerate}
        \item $d$ is an $\mathbb{N}$-linear combination of the weights $a_i$ for $i \in I$, or
        \item there are at least $|I|$ numbers $j \notin I$ such that $d-a_j$ is an $\mathbb{N}$-linear combination of the numbers $a_i$ with $i \in I$.
    \end{enumerate}
\end{enumerate}
\end{proposition}

A key property of well-formed quasi-smooth hypersurfaces $X$ is that the adjunction formula holds, so that $K_X = \mathcal{O}_X(d- a_0 - \cdots - a_N)$ \cite[section 6.14]{Iano-Fletcher}.  Also, the volume of the ample
Weil divisor $\mathcal{O}_X(1)$ is $d/(a_0 \cdots a_n)$.

We now show that all quasi-smooth hypersurfaces
of a given degree have the same singularities, \'etale-locally (and hence
up to local analytic isomorphism).
The argument shows more generally that for any smooth
proper family of complex Deligne-Mumford stacks, their coarse moduli spaces
have locally constant singularities. It follows, for example,
that if one quasi-smooth hypersurface
of some degree $d$ is canonical (resp.\ terminal), then every
quasi-smooth hypersurface of degree $d$ is canonical
(resp.\ terminal).

\begin{lemma}
\label{samesing}
Let $Y$ be a weighted projective space over $\C$, and let $d$
be a positive integer. Suppose that there
is a quasi-smooth hypersurface of degree $d$ in $Y$. Then all quasi-smooth
hypersurfaces of degree $d$ have the same singularities, \'etale-locally.
\end{lemma}

\begin{proof}
Let $X\to B$ be the family of all quasi-smooth hypersurfaces
of degree $d$ in $Y$. Thus $B$ is a Zariski open subset
of the projective space of all hypersurfaces of degree $d$.
Let $\pi\colon W\to B$ be the corresponding
family of affine cones minus their origins. Then $X$ is the quotient
variety $W/H$, where $H:=\mathbb{G}_m$
acts on $W$ with finite stabilizer groups. By definition of quasi-smoothness,
$W$ is smooth over $B$.

Let $w$ be a complex point of $W$, and let $x$ be its image
in $X$ and $b$ its image in $B$.
Then the stabilizer group
of $w$ in $H$ is the subgroup $\mu_n$ of $n$th roots of unity,
for some positive integer $n$.
By Alper-Hall-Rydh's relative version of Luna's \'etale
slice theorem, using that $W\to B$ is smooth,
there is a vector bundle $V$ over an \'etale neighborhood $U$ of $b$
with an action of $\mu_n$ on $V$ such that $W\to B$ is \'etale-locally
isomorphic near $w$ to $(H\times V)/\mu_n\to U$, compatibly
with the action of $H$ \cite[Theorem 20.4(4)]{AHR}.
Therefore, $X=W/H$ is \'etale-locally isomorphic over $B$
to $V/\mu_n$, a family of vector spaces divided by $\mu_n$.

The point is that representations of $\mu_n$ are locally constant,
up to isomorphism. So the singularity of the hypersurface $X_b$
at $x$ is \'etale-locally isomorphic to the singularity of some
point on every nearby hypersurface $X_c$. Conversely, suppose
we have a sequence of points $x_i$ in $X$ such that their images
$c_i$ in $B$ approach $b$, and such that the points $x_i$
all have isomorphic singularities in their fibers $X_{c_i}$.
By properness of $X\to B$,
we can assume after passing to a subsequence that the points $x_i$
approach some point $x$ in $X$. But then the local triviality above
implies that the singularity of the points $x_i$ in their fibers
also occurs at some point in $X_b$ (not necessarily at $x$).
Thus the set of singularities that occur on $X_b$ is locally
constant as a function of $b$ in $B$. Since $B$ is connected,
all the hypersurfaces $X_b$ have the same singularities.
\end{proof}

The condition in Proposition \ref{qsmooth} for quasi-smoothness
is always satisfied when all weights of the weighted projective space $Y$ divide the degree $d$.  When this holds, the sheaf $\mathcal{O}_Y(d)$ is a line bundle and is basepoint-free. In this case, if $Y$ is canonical (resp.\ terminal),
then so is a general hypersurface $X$ of degree $d$, by Koll\'ar's
Bertini theorem \cite[Proposition 7.7]{Kollarsing}.
(This uses again that we are in characteristic zero.)
More generally, still assuming that $d$ is a multiple
of all weights, $X$ is canonical (resp.\ terminal)
if $Y$ has the same property
outside the coordinate points, since $X$ misses those points.
However, we'll often have to deal with examples where
not all weights divide $d$, so that $\mathcal{O}_Y(d)$ is not basepoint-free.  In those cases, we use the following result to determine the type of singularities a quasi-smooth hypersurface has.  It is a generalization of results of Iano-Fletcher for surfaces and threefolds \cite[Theorems 13.1 and 14.4]{Iano-Fletcher}:

\begin{proposition}
\label{singcriterion}
Let $X_d \subset \mathbb{P}(a_0,\ldots,a_N)$ be a quasi-smooth hypersurface of degree $d$. Let $p$ be a point in the stratum $U_I$, $|I| = k+1$, and let $r = \gcd(a_i: i \in I)$.
\begin{enumerate}
    \item If $U_I$ is not in the base locus of $\mathcal{O}_X(d)$, then a neighborhood of $p$ in $X$ is analytically isomorphic to a quotient singularity of type $\frac{1}{r}(a_i: i \notin I) \times \mathbb{A}^{k-1}$, where $j \in I$.
    \item If $U_I$ is in the base locus of $\mathcal{O}_X(d)$, then $d$ is not an $\mathbb{N}$-linear combination of $a_i$ with $i \in I$, so there exists $j \notin I$ such that $r$ divides $d-a_j$ by quasi-smoothness. In this case, a neighborhood of $p$ in $X$ is analytically isomorphic to a quotient singularity of type 
    $\frac{1}{r}(a_i: i \notin I, i \neq j ) \times \mathbb{A}^k$.
\end{enumerate}
\end{proposition}

The simplest application of the second part of the proposition is to the case where we only have one weight $a_0$ that doesn't divide $d$.  Then the only $I$ to which (2) applies is $I = \{0\}$.  In this situation, the coordinate point of $a_0$ is a basepoint of $\mathcal{O}_X(d)$, and so a quasi-smooth $X$ of this degree in $Y$ has a singularity of type
$$\frac{1}{a_0}(a_1,\ldots ,\widehat{a_j},\ldots ,a_N)$$
at that point. 

\begin{proof}
By Lemma \ref{samesing}, we can assume that the hypersurface $X$
is general in its degree.  We'll focus on calculating the singularities of a general $X$ on $U_I$; this determines the singularities of every quasi-smooth hypersurface by the lemma. Note that the quotient singularity type in both (1) and (2) does not depend on which index we choose satisfying the given condition.

By Proposition \ref{checksing}, the weighted projective space $Y$ has a singularity of type $\frac{1}{r}(a_i: i \notin I) \times \mathbb{A}^k$ at $p$. The idea is that the hypersurface $X$ is locally an affine slice of this singularity, and hence has the same type except with possibly one weight removed.  In case (1), the general $X$ is transverse to the stratum $U_I$, so locally we can take $X$ to be $\{x_j = 0\}$ for $x_j$ some coordinate of $\mathbb{A}^k$, meaning that $j \in I$.  This proves part (1).

Next, we need to check singularities on the base locus. The base locus of $\mathcal{O}_X(d)$ is then the union of the $U_I$ corresponding to those $I$ with $d$ not a linear combination of weights in $I$.  Indeed, for a point $p \in U_I$, any homogeneous $f$ of degree $d$ vanishes at $p$ because there is no monomial of degree $d$ containing only variables $x_i$, $i \in I$.  Conversely, if $p \notin U_I$ for any $I$ of this sort, there is some monomial of degree $d$ not vanishing at $p$, and it is not in the base locus. 

Since $X$ is quasi-smooth, there is a set $J$ as in \ref{qsmooth}(2).  For each $j \in J$, there is a monomial $g(x_I) x_j$ of degree $d$, where $g(x_I)$ is a monomial in $\{x_i: i \in I\}$.  This implies that for $f$ general of degree $d$, $\frac{\partial f}{\partial x_j}$ is a polynomial that doesn't vanish identically on $U_I$.  If it doesn't vanish at $p$, the inverse function theorem implies that we can take the remaining variables as coordinates of $X$, so the local picture is the same as the quotient singularity above with $a_j$ removed, just as Proposition \ref{checksing} states. 

However, $\frac{\partial f}{\partial x_j}$ may vanish on a codimension 1 subset of $U_I$ when $f$ is general.  This is why we need $|I|$ different indices $j \in J$ to choose from.  The key fact is the following:
\begin{lemma}
Let $\mathbb{P}$ be any weighted projective space and consider a linear system $L$ on $\mathbb{P}$ consisting of all linear combinations of a set of monomials.  Then the base locus of $L$ (if nonempty)
is a union of coordinate linear subspaces of $\mathbb{P}$.
\end{lemma}

\begin{proof}
The linear system is invariant under the torus $T\cong (\mathbb{G}_m)^N$
acting on $\mathbb{P}$,
and so its zero set in $\mathbb{P}$ is $T$-invariant as well as closed.
\end{proof}
 
To apply this, consider the set $S$ of the monomials $g(x_I)$ of degree $d-a_j$ that appear in $\frac{\partial f}{\partial x_j}$.  Restricting to $V_I$, the partial derivative becomes an arbitrary linear combination of monomials in $S$.  As $f$ varies over all polynomials of degree $d$, the corresponding linear system of restrictions of $\frac{\partial f}{\partial x_j}$ has no base points on $U_I$ (if they exist anywhere on $V_I$, the closure $\overline{U_I}$ of $U_I$, they are in smaller strata).

Now, take a general $f$ and consider all $\frac{\partial f}{\partial x_j}$ for $j \in J$ restricted to $V_I$.  The dimension of $V_I$ is $|I| - 1$, while $|J| = |I|$, so either the $\frac{\partial f}{\partial x_j}$ don't have a common vanishing point on $V_I$, or their common vanishing set lies on $V_I \setminus U_I$ (i.e. on the coordinate planes of $V_I$). This completes the proof, because $f$ will have some partial derivative in $J$ nonvanishing at any $p \in U_I$.
\end{proof}

We next prove two useful restrictions on the singularities
of weighted projective hypersurfaces. The first result
(Corollary \ref{calabiyau}), due to Artebani-Comparin-Guilbot
\cite[Proposition 2.12]{ACG}, says that we get canonical singularities
very easily in the Calabi-Yau case. We start with the well-known
(cf.\ \cite[Remark 1.8]{Reid3}):

\begin{lemma}
\label{cartier}
Let $X$ be a klt variety with the property that $K_X$ is a Cartier divisor.  Then $X$ has canonical singularities.
\end{lemma}

\begin{proof}
Let $f: W \rightarrow X$ be a resolution of singularities.  Then we have 
$$K_W = f^*K_X + \sum_j b_j E_j,$$
where the $E_j$ are the exceptional divisors of $f$.  Since $X$ is klt, each
coefficient $b_j$ is greater than $-1$.
Because $K_X$ is Cartier, each
$b_j$ is an integer.  Therefore, $b_j \geq 0$ for all $j$ and $X$ has canonical singularities.
\end{proof} 

In the context of weighted projective spaces, this leads to the following corollaries.

\begin{corollary}
\label{calabiyau}
Let $X \subset \mathbb{P}(a_0,\ldots ,a_N)$ be a quasi-smooth hypersurface of degree
$\sum_j a_j$, so that $K_X=\mathcal{O}_X$.
Then $X$ has canonical singularities.
\end{corollary}

\begin{proof}
Since $X$ is quasi-smooth, it has only quotient singularities and is therefore klt.  Moreover, $K_X = \mathcal{O}_X(d - \sum_i a_i) = \mathcal{O}_X$ is Cartier, so Lemma \ref{cartier} applies.
\end{proof} 

\begin{corollary}
\label{wpscy}
Let $Y = \mathbb{P}(a_0,\ldots ,a_N)$ be a well-formed weighted projective space with the property that each weight $a_j$ divides the sum $\sum_j a_j$.  Then $Y$ has canonical singularities.
\end{corollary}

\begin{proof}
$Y$ has only quotient singularities, so it is klt.  Furthermore, $K_Y = \mathcal{O}_Y\left(-\sum_j a_j\right)$ is a line bundle because the sum is a multiple of each weight.
\end{proof}

Another related trick may be used on certain subsets of weights.

\begin{lemma}
\label{weightsubset}
Let $\frac{1}{r}(b_1,\ldots,b_s)$ be a well-formed quotient singularity with the property that some nonempty subset $I \subset \{b_1,...,b_s\}$ has sum congruent to $0 \bmod r$ and $\gcd(I \cup \{r\}) = 1$.  Then the singularity is canonical.
\end{lemma}

\begin{proof}
Since the singularity is well-formed, we may apply the Reid-Tai criterion.  Let $1 \leq i < r$ be an integer and consider
$$\sum_{j=1}^s ib_j  \bmod r = \sum_{j \in I} ib_j  \bmod r + \sum_{j \notin I} ib_j  \bmod r.$$
The first sum on the right-hand side must be a multiple of $r$ because the sum of weights in $I$ is a multiple of $r$.  Moreover, it cannot be zero because then each $ib_j$ would be a multiple of $r$.  This would imply that all $b_j$, $j \in I$ share a common factor with $r$, a contradiction.  Therefore, the right-hand side is at least $r$ and the singularity is canonical.
\end{proof}

The second point is that we need not distinguish between canonical and terminal under some circumstances.
For each point $x$ on a complex variety $X$ with $K_X$ $\Q$-Cartier,
some neighborhood $N$ of $x$ has an ``index-1 cover''
$Y\to N$
(unique up to isomorphism \'etale-locally near $x$),
which makes the canonical divisor
Cartier \cite[Section 3.6]{Reidyoung}.
We use the following result of Reid's \cite[Proposition 3.1(II)]{Reidminimal}:

\begin{theorem}
\label{index1}
A canonical singularity whose index-1 cover is terminal must
be terminal.
\end{theorem}

The case we need is that a canonical singularity whose index-1 cover
is smooth must be terminal. This case can also be checked directly
from the Reid-Tai criterion for quotient singularities.

\begin{corollary}
\label{canterm}
Let $X$ be a well-formed, quasi-smooth subvariety of a weighted
projective space such that $K_X = \mathcal{O}_X(1)$ or $K_X = \mathcal{O}_X(-1)$.  If $X$ is canonical, then it is terminal.
\end{corollary}

\begin{proof}
Let $X$ be a well-formed, quasi-smooth subvariety of a weighted
projective space $Y$. Let $U$ be the affine cone over $X$ minus
the origin; then $U$ is smooth,
and $X$ is the quotient of $U$ by $\mathbb{G}_m$. The action
of $\mathbb{G}_m$ on $U$ is proper, and it is free in codimension 1
by the well-formedness of $X$. Let $x$ be a complex
point of $X$, and $u$ a lift of this point to $U$. The stabilizer
subgroup of $u$ is $\mu_r\subset \mathbb{G}_m$ for some positive integer $r$.
By Luna's \'etale slice theorem, an \'etale neighborhood $N$ of $x$ in $X$
is the quotient of a smooth
variety $M$ by $\mu_r$, and the action is free in codimension 1. So the local
class group of $X$ at $x$ is cyclic of order $r$, generated by
$\mathcal{O}_X(1)$, using that the corresponding line bundle on $U$
is the trivial bundle with $\mathbb{G}_m$ acting by scalars.

Suppose now that $K_X$ is $\mathcal{O}_X(1)$ or $\mathcal{O}_X(-1)$.
Then $K_X$ generates the local class group of $X$ at $x$,
and so the index-1 cover of $X$ at $x$ is precisely $M\to N$.
Here $M$ is smooth. By Theorem \ref{index1}, if $X$ is canonical,
then it is terminal.
\end{proof}

\section{Proof of Main Results}
In this section, we will find examples of hypersurfaces of various types with excellent asymptotics in high dimensions.  For several of the problems, we conjecture that the examples we find are optimal.  The examples we construct will all be based on Sylvester's sequence $s_0, s_1, s_2,\ldots $, defined by $s_0 = 2$ and $s_m = s_{m-1}(s_{m-1}-1)+1$ for $m \geq 1$.  It follows that $s_m = s_0 \cdots s_{m-1} + 1$, and hence that the integers in Sylvester's sequence
are pairwise coprime. The first few values
are $s_0=2$, $s_1=3$, $s_2=7$, $s_3=43$, $s_4=1807$.

The Sylvester sequence grows doubly exponentially; in fact, $s_m = \lfloor E^{2^{m+1}} + \frac{1}{2} \rfloor$ for a constant $E \doteq 1.264$ \cite[equations
2.87 and 2.89]{GK}.  We'll frequently make use of the estimate $s_n > 2^{2^{n-1}}$.  Also, the sums of the reciprocals of the $s_j$ converge to $1$ more quickly than any other sequence of unit fractions \cite{Soundararajan}.
Namely, we have
$$\frac{1}{s_0} + \frac{1}{s_1} + \cdots + \frac{1}{s_{m-1}}
= 1 - \frac{1}{s_m-1} = 1 - \frac{1}{s_0 \cdots s_{m-1}}.$$

Table \ref{summary} summarizes the examples we construct
for each of the 8 problems. The table shows (approximately) the
volume or the first nonvanishing space of sections,
in terms of Sylvester's sequence $s_j$. Note that $s_{j+1}$ is roughly
$s_j^2$. Our examples 1a) and 2a) are conjecturally optimal, and
example 1b) is conjecturally optimal among quasi-smooth hypersurfaces.
\begin{table}
\centering
\renewcommand{\arraystretch}{1.2}
\begin{tabular}{|c|c|c|c|c|}
\hline
& 1. Canonical CY & 2. Terminal Fano & 3. General type & 4. Terminal CY\\
\hline
a) Volume & $1/(s_n)^{2n-1}$ & $1/(s_{n-1})^{2n-3}$ & $1/(s_{\lfloor (n+1)/2
\rfloor})^n$ & $1/(s_{\lfloor n/2\rfloor})^n$\\
\hline
b) First $H^0\neq 0$ & $s_n$ & $s_{\lfloor (n-1)/2\rfloor}$ &
$s_{\lfloor (n-1)/2\rfloor}$ & $(s_{\lfloor (n-4)/2\rfloor})^{3/2}$\\
\hline
\end{tabular}
\caption{Summary of our examples, approximately stated.}
\label{summary}
\end{table}

\subsection{Canonical Calabi-Yau Varieties}

In light of Corollary \ref{calabiyau}, the problem of finding extreme behavior among weighted projective hypersurfaces is easiest for canonical Calabi-Yau varieties, where we need only check that a given example is well-formed and quasi-smooth.  For each positive integer $n$, Birkar showed that there is a bound $M = M(n)$ such that for all ample Weil divisors $A$ on klt Calabi-Yau $n$-folds, the linear system $|\ell A|$ defines a birational embedding into projective
space for all $\ell \geq M$ \cite[Corollary 1.4]{Birkar}. This implies that there is a positive lower bound on volume in each dimension for problem 1a), and an upper bound on the number of vanishing $H^0$ groups in problem 1b).  We'll first tackle 1a) by producing a Calabi-Yau hypersurface $X$ for which the ample Weil divisor $\mathcal{O}_X(1)$ has small volume. We conjecture
that this example is optimal, as discussed below.

\begin{proposition}
\label{cyex1}
For each positive integer $n$, let $d = (2s_n - 3)(s_n-1)$.  Then a general hypersurface $X$ of degree $d$ in the weighted projective space $Y = \mathbb{P}^{n+1}(d/s_0,\ldots,d/s_{n-1},s_n-1,s_n-2)$ is quasi-smooth and Calabi-Yau with canonical singularities.  The ample Weil divisor $\mathcal{O}_X(1)$ has volume less than $1/2^{2^n}$ for $n\geq 2$.
\end{proposition}

\begin{proof}
Since $s_n - 1 = s_0 \cdots s_{n-1}$, each weight is an integer and all but the last weight divide the degree. The last weight $s_n-2$, however, divides $d - (s_n - 1) = (2s_n - 4)(s_n - 1)$.  Therefore, a general hypersurface $X$ of degree $d$ in $Y$ is quasi-smooth, using Proposition \ref{qsmooth}.  Because $s_n-2$ and $s_n-1$ are relatively prime, we need only check the well-formedness condition when one of these two weights is removed.  However, the greatest common divisor of $d/s_0,\ldots d/s_{n-1}$ is $d/(s_0 \cdots s_{n-1}) = 2s_n - 3$.  This is relatively prime to both $s_n-2$ and $s_n-1$, so the required condition holds.   Furthermore, the sum of the weights is 
$$d(1/s_0 + \cdots + 1/s_{n-1}) + (s_n - 1) + (s_n - 2) = d(1-1/(s_n - 1)) + (s_n - 1) + (s_n - 2) $$
$$=  d - (2s_n - 3) + (s_n - 1) + (s_n - 2) = d,$$

so $X$ is Calabi-Yau.  Finally, the volume of the ample Weil divisor $\mathcal{O}_X(1)$ is

$$\mathrm{vol}(\mathcal{O}_X(1)) = \frac{d}{(d/s_0) \cdots (d/s_{n-1})(s_n - 1)(s_n - 2)} = \frac{s_0 \cdots s_{n-1}}{d^{n-1}(s_n-1)(s_n-2)} $$$$= \frac{1}{d^{n-1}(s_n-2)} = \frac{1}{(2s_n-3)^{n-1}(s_n - 1)^{n-1}(s_n-2)}.$$

As a crude estimate, this last volume is less than $1/(s_n)^{2n-1} < 1/1.264^{(2n-1)2^{n+1}}$, where $1.264$ is less than $E \doteq 1.264$ above.  Since $(2n-1)2^{n+1} = 2(2n-1)2^n$ and $1.264^3 > 2$, this last volume is less than $1/2^{2^n}$ when $n \geq 2$. \end{proof}

For $n = 1$, this example gives $X_6 \subset \mathbb{P}^2(3,2,1)$ of volume 1, an elliptic curve $E$ embedded into Proj of its section ring $R(E,\mathcal{O}(P))$, where $P$
is the origin of $E$.  When $n = 2$, we obtain the surface $X_{66} \subset \mathbb{P}^3(33,22,6,5)$ for which $\mathcal{O}_X(1)$ has volume $1/330$, which is minimal among all canonical K3 surfaces
with an ample Weil divisor \cite[Computation 1.1]{Brown}.  The next two examples are the threefold $X_{3486} \subset \mathbb{P}^4(1743,1162,498,42,41)$ with volume $1/498240036$ and the fourfold 
$$X_{6521466} \subset \mathbb{P}^5(3260733,2173822,931638,151662,1806,1805)$$
with volume approximately $2.0\times 10^{-24}$.  Both of these are known to have minimal volume among quasi-smooth Calabi-Yau hypersurfaces in their respective dimensions, by Brown and Kasprzyk
(see \cite{database} for $n  = 3$ and \cite[Section 1.2]{Brown-Kasprzyk} for $n = 4$). These results motivate our conjecture:

\begin{conjecture}
\label{1aconj}
Let $n$ be a positive integer and $X$ the $n$-fold in Proposition \ref{cyex1}.  Then $\mathcal{O}_X(1)$ has minimal volume among all ample Weil
divisors on canonical $n$-folds $X$ which are Calabi-Yau
in the sense that $K_X\sim_{\Q} 0$.
\end{conjecture}

Next, we'll consider problem 1b), namely the requirement that $h^0(X,\ell A) = 0$ for some ample divisor $A$ and for $1 \leq \ell < M$ with $M$ large.  Once again, we can conjecture an optimal hypersurface example using Sylvester's sequence:

\begin{proposition}
\label{cyex2}
For a positive integer $n$, let $d=(s_{n-1}-1)(3s_{n-1}-4)^2$. Then a general hypersurface $X$ of degree $d$ in the weighted projective space $Y = \mathbb{P}^{n+1}(d/s_0,\ldots,d/s_{n-2},(s_{n-1}-1)(3s_{n-1}-4),(s_{n-1}-1)(3s_{n-1}-5),3s_{n-1}^2-9s_{n-1} + 7)$
is quasi-smooth and Calabi-Yau with canonical singularities.   The spaces
of sections $H^0(X,\mathcal{O}_X(\ell))$ vanish for $1\leq \ell < M$, where $M = 3s_{n-1}^2-9s_{n-1}+7 > 2^{2^{n-1}}$ if $n\geq 2$. \end{proposition}
\begin{proof}
All but the last two weights divide the degree. Let $a=(s_{n-1}-1)(3s_{n-1}-4)$, $b=(s_{n-1}-1)(3s_{n-1}-5)$, and $c = 3s_{n-1}^2-9s_{n-1}+7$.  First, we'll show that $Y$ is well-formed.  Observe that $\gcd(a,b,c) = 1$.  This is because the greatest common factor of $a$ and $b$ is $s_{n-1}-1$, while $c = (s_{n-1}-1)(3s_{n-1}-6) + 1$.  Therefore, it remains to see that when one of these three weights is removed, the gcd of all other weights is still $1$.  We have $\gcd(d/s_{0},\ldots ,d/s_{n-2}) = (3s_{n-1}-4)^2$, which is coprime to $b$.  Finally, if $b$ is removed, the gcd of all weights besides $c$ is $3s_{n-1}-4$ and $c - (3s_{n-1}-4)(s_{n-1}-2)= s_{n-1} -1$.  This last expression is coprime to $3s_{n-1}-4$, completing the argument.

We have $b=(d-a)/(3s_{n-1}-4)$ and $c=(d-b)/(3s_{n-1}-3)$. Thus $X$ is quasi-smooth. Finally, taking the sum of all weights gives
$$d(1/s_0 + \cdots + 1/s_{n-2}) + a + b + c = d(1-1/(s_{n-1}-1)) + 9s_{n-1}^2 - 24s_{n-1} + 16 $$
$$= d - (3s_{n-1}-4)^2 + (3s_{n-1}-4)^2 = d.$$
Therefore the hypersurface is Calabi-Yau and is canonical by Corollary \ref{calabiyau}.
\end{proof}

In low dimensions this example becomes $X_{50}\subset \mathbb{P}^3(25,10,8,7)$ for $n=2$, $X_{1734}\subset \mathbb{P}^4(867,578,102,96,91)$ for $n=3$, and
$$X_{656250}\subset \mathbb{P}^5(328125,218750,93750,5250,5208,5167)$$
for $n=4$.  In each of these dimensions, $X$ has the largest possible bottom weight for any quasi-smooth Calabi-Yau hypersurface (see \cite{database} and \cite[Section 1.2]{Brown-Kasprzyk}).  In general, we make the following conjecture.

\begin{conjecture}
\label{1bconj}
Let $n$ be a positive integer and $X$ the $n$-fold in Proposition \ref{cyex2}.
Then $X$ has the largest possible positive integer $M$ such that
$H^0(X,\mathcal{O}_X(\ell))$ vanishes for $1 \leq \ell < M$,
among all quasi-smooth Calabi-Yau hypersurfaces of dimension $n$.
\end{conjecture}

Note, however, that this value $M$ is not optimal among all canonical Calabi-Yau varieties with an ample Weil divisor. We see this in dimension 2. Namely, Iano-Fletcher
found that a general complete intersection $X_{24,30} \subset \mathbb{P}^4(15,12,10,9,8)$ of codimension 2 is a K3 surface with canonical singularities, and $h^0(X,\mathcal{O}_X(\ell)) = 0$ for $1 \leq \ell < 8$ \cite[Section 13.8]{Iano-Fletcher}. In contrast, the maximum bottom weight in the hypersurface case is $7$.

Putting together Propositions \ref{cyex1} and \ref{cyex2}, we've now proven the following theorem.

\begin{theorem}
\label{cancytheorem}
\begin{enumerate}
    \item For every integer $n\geq 2$, there is a canonical Calabi-Yau variety $X$ of dimension $n$ with an ample Weil divisor $A$ such that $\mathrm{vol}(A) < 1/2^{2^{n}}$.
    \item For every integer $n\geq 2$, there is
a canonical Calabi-Yau variety $X$ of dimension $n$ with an ample Weil divisor $A$ such that $H^0(X,\ell A) = 0$ for $1 \leq \ell < 2^{2^{n-1}}$.
\end{enumerate}
\end{theorem}

\subsection{Terminal Fano Varieties}
We now turn to terminal Fano varieties.  The class of klt Fano varieties
of given dimension has no positive lower bound on volume
\cite[Example 2.1.1]{HMXacc}. But Birkar showed that there is a lower bound
for $\epsilon$-lc Fano varieties, hence for terminal Fano
varieties \cite[Theorem 1.2]{Birkarcomp}.  The number of vanishing anti-plurigenera, meanwhile, is bounded among all klt Fano varieties
of each dimension, by Birkar's theorem on boundedness
of complements \cite[Theorem 1.1]{Birkarcomp}.  So problems 2a) and 2b) have some bound in each dimension.  In the case of problem 2a), we can conjecture the optimal example, obtained by adding an extra weight of $1$ to the Calabi-Yau example 1a). See the evidence below.

For problem 2b), we find a terminal Fano $n$-fold
with $H^0(X,-mK_X)=0$ for all $1\leq \ell < M$ with $M$ roughly
$2^{2^{n/2}}$ (Theorem \ref{fanotheorem}).
In the broader class of klt Fano $n$-folds,
Totaro and Wang gave examples with $M$ roughly
$2^{2^n}$ \cite[Theorem 5.1]{Totaro-Wang}.

\begin{proposition}
\label{fanoex1}
For each integer $n \geq 2$, let $d = (2s_{n-1} - 3)(s_{n-1}-1)$.  Then a general hypersurface $X$ of degree $d$ in the weighted projective space $Y = \mathbb{P}^{n+1}(d/s_0,\ldots,d/s_{n-2},s_{n-1}-1,s_{n-1}-2,1)$ is
a terminal Fano variety, and $\vol(-K_X)< 1/2^{2^{n}}$ when $n \geq 3$.
\end{proposition}

\begin{proof}
The proof is nearly identical to that of Proposition \ref{cyex1} with $n$ replaced by $n-1$. Here $X$ is quasi-smooth with $K_X=\mathcal{O}_X(-1)$.
Adding a weight 1
to the weights in Proposition \ref{cyex1} ensures by the Reid-Tai criterion that all canonical singularities become terminal (this is also a consequence
of Corollary \ref{canterm}).  We may still use $1/2^{{2^n}}$ as the volume bound because our estimate $1/1.264^{(2n-3)2^{n}}$ still becomes less than $1/2^{{2^n}}$ when $n \geq 3$.
\end{proof}

When $n = 2$, this example gives $X_6 \subset \mathbb{P}^3(3,2,1,1)$ for $n=2$, which is the natural embedding of a smooth del Pezzo surface of degree 1 via its anticanonical ring.  Since terminal surfaces are smooth, its (anticanonical) volume of 1 is minimal.  In dimension $3$,
we obtain $X_{66} \subset \mathbb{P}^4(33,22,6,5,1)$ of volume 1/330,
which is minimal among all terminal Fano 3-folds,
by J.~Chen and M.~Chen \cite{Chen08}.
Finally, when $n = 4$, we obtain
$X_{3486} \subset \mathbb{P}^5(1743,1162,498,42,41,1)$, with volume $1/498240036$. This volume is minimal among all quasi-smooth terminal 4-fold hypersurfaces with $K_X=\mathcal{O}_X(-1)$,
by Brown and Kasprzyk \cite[Section 1.3]{Brown-Kasprzyk}. These results
justify the conjecture:

\begin{conjecture}
\label{2aconj}
For an integer $n\geq 2$, let $X$ be the variety of dimension $n$
in Proposition \ref{fanoex1}.  Then $X$ has minimal anticanonical volume among all terminal Fano $n$-folds. \end{conjecture}

We can also find Fano hypersurfaces with many vanishing groups $H^0(X,-\ell K_X)$, though these examples are not optimal. (In dimension 3, Iano-Fletcher
found the complete intersection 3-fold $X_{12,14}
\subset \mathbb{P}^5(7,6,5,4,3,2)$, a terminal variety
with $K_X=\mathcal{O}_X(-1)$. This has $H^0(X,-K_X)=0$, unlike
any terminal Fano hypersurface of dimension 3
\cite[section 16.7]{Iano-Fletcher}.)

 The example below is a slight modification of the example obtained from 1a) by repeating each weight twice.

\begin{proposition}
\label{fanoex2}
\begin{enumerate}
    \item Let $n = 2m+1$ be an odd integer at least 3, and let $d=(2s_m-3)(s_m-1)$. Then a general hypersurface $X$ of degree $2d$ in the weighted projective space $Y =\mathbb{P}^{n+1}((d/s_0)^{(2)},\ldots,(d/s_{m-1})^{(2)},2(s_m-1),s_m-1,s_m-2)$ is quasi-smooth and Fano with terminal singularities. Such an $X$ satisfies $H^0(X, -\ell K_X)=0$ for $1\leq \ell < M$, where $M = s_m-2 > 2^{2^{(n-3)/2}}$.
    \item Let $n = 2m+2$ be an even integer at least 6, and let $d=(2s_m-3)(s_m-1)$. Then a general hypersurface $X$ of degree $2d$ in the weighted projective space $Y =\mathbb{P}^{n+1}((d/s_0)^{(2)},\ldots,(d/s_{m-2})^{(2)},d/s_{m-1},(d/2s_{m-1})^{(2)},2(s_m-1),s_m-1,s_m-2)$ is quasi-smooth and Fano with terminal singularities. Such an $X$ satisfies $H^0(X, -\ell K_X)=0$ for $1\leq \ell < M$ where $M = s_m-2 > 2^{2^{(n-4)/2}}$.
\end{enumerate}
\end{proposition}
To get from the odd-dimensional case to the even dimensional case for the same $m$, we split one copy of the weight $d/s_{m-1}$ in two (it is even since it is divisible by $s_0 = 2$).  We'll prove only (1), since (2) directly follows.

\begin{proof}
All weights but the last divide the degree, while $s_m-2$ divides $2d - 2(s_m-1) = (4s_m-8)(s_m-1)$.  It follows that a general $X$ is quasi-smooth.  Since it contains the weights of the Calabi-Yau example of Proposition \ref{cyex1}, the weighted projective space is well-formed.  The sum of the weights is 
$$2(d/s_0 + \cdots + d/s_{m-1}) + 2(s_m-1) + s_m-1 + s_m-2 = 2d(1-1/(s_m-1)) + 4s_m-5 = 2d + 1.$$
Thus, $X_{2d}$ is Fano with dimension $n=2m+1$. 

We'll next show that $X$ has canonical singularities.  We already know that the weighted projective space $Y' := \mathbb{P}^{m+1}(d/s_0,\ldots,d/s_{m-1},s_m-1,s_m-2)$ has canonical singularities, again by Proposition \ref{cyex1}.  Since the weights of $Y'$ are a subset of those of $Y$, the Reid-Tai criterion is satisfied for each weight of $Y$ which is also a weight of $Y'$.  This includes the criterion for $X$ at the coordinate point $p$ for the weight $s_m-2$, since only the weight $2(s_m-1)$ not belonging to $Y'$ is removed, by Proposition \ref{singcriterion}.  This is the only basepoint of the sheaf $\mathcal{O}_Y(2d)$.
In summary, we have that $X$ and $Y$ are both canonical away from the coordinate point
with weight $2(s_m-1)$.  Because a general hypersurface misses this point, we may conclude that $X$ is canonical. By Corollary \ref{canterm}, $X$ is terminal.
\end{proof}

The following statement
summarizes Propositions \ref{fanoex1}, \ref{fanoex2}, and the discussion
above.

\begin{theorem}
\label{fanotheorem}
\begin{enumerate}
    \item For every integer $n\geq 3$, there is a terminal Fano variety $X$ of dimension $n$ with $\mathrm{vol}(-K_X) < 1/2^{2^n}$.
    \item For every integer $n\geq 3$, there is a terminal Fano variety $X$ of dimension $n$ with the property that $H^0(X,-\ell K_X) = 0$ for $1 \leq \ell < 2^{2^{(n-4)/2}}$.
\end{enumerate}
\end{theorem}

\subsection{Varieties of General Type}

We'll next find examples for 3a) and 3b) to prove Theorem \ref{main}:
that is, varieties of general type with small volume
or many vanishing plurigenera.
There is some bound for these problems in each dimension,
by the results of Hacon-M\textsuperscript{c}Kernan, Takayama,
and Tsuji mentioned in the introduction. (Note that the plurigenera
of a variety with canonical singularities
are equal to the plurigenera of any resolution of singularities.
So it is a matter of choice whether to consider varieties
with canonical singularities and ample canonical class
or smooth projective varieties of general type.)

The constructions will be modifications
of Calabi-Yau examples above.  First is an example of small volume.

\begin{lemma}
\label{basiccyex}
For each natural number $m$, let $d = s_{m+1}-1$.  Then the weighted projective space $Y' = \mathbb{P}^{m+1}(d/s_0,\ldots,d/s_{m},1)$ is well-formed and has canonical singularities.
\end{lemma}

\begin{proof}
First, $\gcd(d/s_0,\ldots,d/s_m) = 1$, so the space is well-formed.  Further, the sum of the weights is $d(1-1/(s_{m+1}-1)) + 1 = d$.  Since all weights divide this sum, $K_{Y'}$ is Cartier and $Y'$ has canonical singularities by Corollary \ref{wpscy}.
\end{proof}

\begin{proposition}
\label{gentypeex1}
Let $n = 2m+1$ be an odd positive integer, and let $d = s_{m+1}-1$.  Then a general hypersurface $X$ of degree $2d$ in the weighted projective space $Y = \mathbb{P}^{n+1}((d/s_0)^{(2)},\ldots,(d/s_m)^{(2)},1)$ is quasi-smooth, with $K_X=\mathcal{O}_X(1)$, and has terminal singularities.  When $n \geq 5$, $\mathrm{vol}(X) < 1/2^{2^{n/2}}$.
\end{proposition}

\begin{proof}
As compared to the example in Lemma \ref{basiccyex}, all weights are repeated
twice except for $1$.  All the required properties follow from the lemma; in particular $K_X = \mathcal{O}_X(1)$ because the sum of the weights is $2d-1$.  Since all weights are repeated from the example in the lemma, the Reid-Tai criterion automatically holds for each coordinate point and $Y$ is canonical.
Since the degree is a multiple of all weights, it follows
that $X$ is quasi-smooth and canonical. Since $K_X = \mathcal{O}_X(1)$, $X$ is in fact
terminal, by Corollary \ref{canterm}.

The volume of this example is
$$\mathrm{vol}(X) = \frac{2d}{(d/s_0)^2 \cdots (d/s_m)^2} = \frac{2d(s_0 \cdots s_m)^2}{d^{2m+2}} = \frac{2}{(s_{m+1}-1)^{2m-1}} = \frac{2}{(s_{(n+1)/2}-1)^{n-2}}.$$
Since $s_{(n+1)/2} > 2^{2^{(n-1)/2}}$, we have that when $n \geq 5$, the volume is smaller than $1/2^{2^{n/2}}$.
\end{proof}

\begin{lemma}
For each natural number $m$, let $d = (s_m-1)(2s_m-1)$.  Then the weighted projective space $Y' = \mathbb{P}^{m+2}(d/s_0,\ldots,d/s_{m-1},(s_m-1)^{(2)},1)$ is well-formed and has canonical singularities.
\end{lemma}

\begin{proof}
For $m=0$, $Y'=\mathbb{P}^2$ and the result is clear. For $m\geq 1$,
$\gcd(d/s_0,\ldots,d/s_{m-1}) = 2s_m-1$ is coprime to $s_m-1$,
and so $Y'$ is well-formed.  The sum of all but the last three weights is $d(1-1/(s_m-1)) = d- (2s_m -1)$, so the sum of all weights is $d$.  Therefore, $Y'$ is canonical by Corollary \ref{wpscy}.
\end{proof}

\begin{proposition}
\label{gentypeex2}
Let $n = 2m+2$ be an even integer at least 2, and let $d = (s_m-1)(2s_m-1)$.  Then a general hypersurface $X$ of degree $2d$ in the weighted projective space $Y = \mathbb{P}^{n+1}((d/s_0)^{(2)},\ldots,(d/s_{m-1})^{(2)},2(s_m-1),(s_m-1)^{(2)},1)$ is quasi-smooth
with terminal singularities and has $K_X=\mathcal{O}_X(1)$.
When $n \geq 4$, $\mathrm{vol}(X) < 1/2^{2^{n/2}}$.
\end{proposition}

\begin{proof}
The weights of $Y$ include those of $Y'$ as a subset and each of them divides $2d$.  Therefore, $X$ is well-formed and quasi-smooth. The adjunction
formula gives that $K_X=\mathcal{O}_X(1)$. The Reid-Tai criterion is again automatic at every weight besides $2(s_m-1)$.  Since a general hypersurface of degree $d$ misses this point, we can conclude that $X$ is canonical.
By Corollary \ref{canterm}, $X$ is in fact terminal.

The volume of this example is 
$$\mathrm{vol}(X) = \frac{2d}{ (d/s_0)^2 \cdots (d/s_{m-1})^2 (s_m-1)^2 2(s_m-1)} = \frac{2d(s_0 \cdots s_{m-1})^2}{2(s_m-1)^{2m+3}(2s_m-1)^{2m}}
$$$$= \frac{1}{(s_m-1)^{2m}(2s_m-1)^{2m-1}}.$$
This last expression is less than 
$$\frac{1}{s_m^{4m-1}} = \frac{1}{s_{(n-2)/2}^{2n-5}}.$$
Since $s_{(n-2)/2} > 2^{2^{(n-4)/2}}$, this last expression is less than $2^{2^{n/2}}$ whenever $2n-5 > 4$, i.e. for $n \geq 4$.
\end{proof}

We'll also write down some examples of general type with many
vanishing plurigenera.

\begin{proposition}
\label{gentypeex3}
\begin{enumerate}
    \item Let $n = 2m + 1$ be an odd integer at least $5$, and let $d = (s_m-1)(2s_m-1)$.  Then a general hypersurface $X$ of degree $2d$ in the weighted projective space $Y = \mathbb{P}^{2m+2}((d/s_0)^{(2)},\ldots,(d/s_{m-1})^{(2)},2s_m-2,s_{m-1}^2,(s_{m-1}-1)^2)$ is quasi-smooth with terminal singularities and has $K_X = \mathcal{O}_X(1)$.  Such an $X$ satisfies $H^0(X, \ell K_X) = 0$ for $1 \leq \ell < M$ where $M = (s_{m-1}-1)^2 \geq 2^{2^{(n-3)/2}}$.
    \item Let $n = 2m + 2$ be an even integer at least $6$, and let $d = (s_m-1)(2s_m-1)$.  Then a general hypersurface $X$ of degree $2d$ in the weighted projective space $Y = \mathbb{P}^{2m+3}((d/s_0)^{(2)},\ldots,(d/s_{m-2})^{(2)},
d/s_{m-1},(d/(2s_{m-1}))^{(2)},2s_m-2,s_{m-1}^2,(s_{m-1}-1)^2)$ is quasi-smooth with terminal singularities and has $K_X = \mathcal{O}_X(1)$.  Such an $X$ satisfies $H^0(X, \ell K_X) = 0$ for $1 \leq \ell < M$ where $M = (s_{m-1}-1)^2 \geq 2^{2^{(n-4)/2}}$.
\end{enumerate}

\end{proposition}

\begin{proof}
We'll first prove part (1). Suppose that $p$ is a prime number dividing all but one of the weights.  Since $s_{m-1}^2$ and $(s_{m-1}-1)^2$ are relatively prime, one of these must be omitted, so that $p$ divides $d/s_0,\ldots,d/s_{m-1}$, and $2s_m-2$.  However, $\gcd(d/s_0,\ldots,d/s_{m-1}) = 2s_m-1$,
which is coprime to $2s_m-2$; so $X$ is well-formed.  To see that $X$ is quasi-smooth, we verify Iano-Fletcher's criterion (Proposition \ref{qsmooth}).  Neither $s_{m-1}^2$ nor $(s_{m-1}-1)^2$ divides $d$, but both divide $2d - (2s_m -2) = 2(s_m-1)(2s_m-2) = 4(s_m-1)^2 = 4s_{m-1}^2(s_{m-1}-1)^2$, so the condition is satisfied for one-element subsets.  For the two-element subset $I = \{s_{m-1}^2,(s_{m-1}-1)^2\}$, we claim that $2d$ is an $\mathbb{N}$-linear combination of these weights.  Indeed, we have $2s_m-1 = s_{m-1}^2 + (s_{m-1}-1)^2$ so $2d = 2(s_m-1)(s_{m-1}^2 + (s_{m-1}-1)^2)$.  Finally, the sum of the weights is 
$$2d(1-1/(s_m-1)) + 2s_m-2 + s_{m-1}^2 + (s_{m-1}-1)^2 $$
$$= 2d - (4s_m-2) + 2s_m - 2 + 2s_m - 1 = 2d -1,$$
so that $K_X = \mathcal{O}_X(1)$.  To show that $X$ is terminal, it suffices to show that it is canonical, by Corollary \ref{canterm}.
First, note that the sum
$d/s_0 + \cdots + d/s_{m-1} + s_{m-1}^2 + (s_{m-1}-1)^2$ is equal
to $d$, and the last two weights are relatively prime.  Since each $d/s_j$ divides this sum, Lemma \ref{weightsubset} shows that the singularity at each coordinate point with weight $d/s_j$ is canonical.  Using the torus action, this means that $Y$ is canonical away from the weighted $\mathbb{P}^2_{x,y,z}$ with weights $a := 2s_m-2$, $b := s_{m-1}^2,$ and $c := (s_{m-1}-1)^2$. Because $b$ and $c$ are coprime, $Y$ is actually smooth at points on this weighted $\mathbb{P}^2$ outside of
two $\mathbb{P}^1$s.  We'll consider these in turn.

On the one-dimensional stratum given by $z = 0$, $x,y \neq 0$, $Y$ has singularities of the form $\frac{1}{s_{m-1}}((d/s_0)^{(2)},\ldots,(d/s_{m-1})^{(2)},c)$ because $\gcd(a,b) = s_{m-1}$ (here we use that $m \geq 2$, so $s_{m-1}$ is odd).  Note that the weight $d/s_{m-1} \equiv -1 \pmod{s_{m-1}}$ because it equals $(s_{m-1}-1)(2s_m-1) = (s_{m-1} - 1)(2s_{m-1}^2 - 2s_{m-1} + 1)$, where the first term has residue $-1$ and the second $+1$.  Also, $c = (s_{m-1}-1)^2 \equiv 1 \pmod{s_{m-1}}$.  Therefore, the sum of these two weights is divisible by $s_{m-1}$, while $c$ is coprime to $s_{m-1}$.  Lemma \ref{weightsubset} again applies to show that the singularity is canonical.

Next, on the one-dimensional stratum given by $y = 0$, $x,z \neq 0$, $Y$ has singularities of the form $\frac{1}{2(s_{m-1}-1)}((d/s_0)^{(2)},\ldots,(d/s_{m-1})^{(2)},b)$ because $\gcd(a,c) = 2(s_{m-1}-1)$.  We'll use a similar method to the above to analyze this singularity.  First, the sum of the first $m-1$ weights $d/s_0 + \cdots + d/s_{m-2}$ (with each only repeated once) is 
\begin{align*}
&\; d(1-1/(s_{m-1}-1)) = (s_m-1)(2s_m-1) - (s_{m-1})(2s_m-1)\\
=&\; (s_m - 1 - s_{m-1})(2s_m - 1) = (s_{m-1}^2-2s_{m-1})(2s_{m-1}^2 - 2s_{m-1} + 1).
\end{align*}
The second factor is congruent to $1 \bmod 2(s_{m-1}-1)$. The first is equal to $(s_{m-1}-1)^2 - 1$, which is congruent to $-1 \bmod 2(s_{m-1}-1)$ because $s_{m-1} - 1$ is even.  The product has residue $-1 \bmod 2(s_{m-1}-1)$.  Similarly, $b = s_{m-1}^2 \equiv 1 \pmod{2(s_{m-1}-1)}$.  Therefore, we've found a subset of the weights with sum divisible by $2(s_{m-1}-1)$, to which $b$ is coprime.  Therefore, the singularity is canonical.

At this point, we've shown that $Y$, and hence $X$, is canonical away from the three coordinate points with weights $a, b$, and $c$.  The weight $a$ divides $2d$, so the general hypersurface $X$ misses the corresponding coordinate point and we needn't check the singularity of $Y$ there.  We'll conclude by looking at the singularities at the basepoints $[0:\cdots :0:1:0]$ and $[0:\cdots :0:0:1]$
of $|\mathcal{O}_Y(2d)|$ (corresponding to weights $b$ and $c$, respectively).  At $[0:\cdots :0:1:0]$, $X$ has a singularity of type $\frac{1}{b}((d/s_0)^{(2)},\ldots,(d/s_{m-1})^{(2)},c)$, where $a$ is omitted.  We claim that the sum of weights $d/s_0 + d/s_0 + d/s_{m-1} + c$ is divisible by $b$.  Indeed,
$$d/s_0 + d/s_0 = d = s_{m-1}(s_{m-1}-1)(2s_{m-1}^2-2s_{m-1}+1) \equiv -s_{m-1} \pmod{b},$$
while $d/s_{m-1} = (s_{m-1} - 1)(2s_{m-1}^2-2s_{m-1}+1) \equiv 3s_{m-1} - 1 \pmod{b}$ and $c = (s_{m-1}-1)^2 \equiv -2s_{m-1}+1 \pmod{b}$.  The sum of these residues is $0 \bmod b$.  Since $\gcd(d/s_0,d/s_0,d/s_{m-1},c,b) = 1$, this shows that the singularity is canonical.

Finally, consider the point $[0:\cdots :0:0:1]$.  The singularity there is $\frac{1}{c}((d/s_0)^{(2)},\ldots,(d/s_{m-1})^{(2)},b)$.  The sum of the first $m-1$ weights (once each), $d/s_0 + \cdots + d/s_{m-2}$, is $(s_{m-1}^2-2s_{m-1})(2s_{m-1}^2 - 2s_{m-1} + 1)$ as calculated above.
The first factor is $-1 \pmod{c}$ and the second is $-2s_{m-1} + 1\pmod{c}$, so the product is $2s_{m-1} - 1 \pmod{c}$.  But $b = s_{m-1}^2 \equiv -2s_{m-1} + 1 \pmod{c}$, so the sum of these weights is divisible by $c$.  As before, their gcd with $c$ is $1$, so Lemma \ref{weightsubset} applies.  This completes the proof that $X$ is canonical.

The hypersurface in part (2) is identical except that we've split one copy of $d/s_{m-1}$ into two copies of $d/(2s_{m-1})$.  The same proof works for this case, where we may replace sums involving $d/s_{m-1}$ with $d/(2s_{m-1}) + d/(2s_{m-1})$ where necessary.
\end{proof}

\begin{proof}[Proof of Theorem \ref{main}]
First, consider part (1).  For any dimension $n \geq 4$, let $W$ be a resolution of singularities of the hypersurface $X$ appearing in either Proposition \ref{gentypeex1} or \ref{gentypeex2} of the appropriate dimension.  Since $X$ is terminal, we have $\mathrm{vol}(W) = \mathrm{vol}(X) < 1/2^{2^{n/2}}$, as required.  In both the examples, the last weight is $1$, so the geometric genus $p_g$ is positive (namely, 1). For $n=3$, the example in Proposition
\ref{gentypeex1} is $X_{12} \subset \mathbb{P}^4(3^{(2)},2^{(2)},1)$
with volume 1/3, which is not good enough. Instead, we can
use the 3-fold $X_{28}\subset \mathbb{P}^4(14,5,4,3,1)$, which is terminal
with $K_X=\mathcal{O}_X(1)$ \cite{database}. This has positive
geometric genus, and its volume
$1/30$ is less than
$1/2^{2^{n/2}}$, as we want.

To prove part (2), note that $|\ell K_W|$ does not give a birational embedding for the resolutions of examples $X$ in Propositions \ref{gentypeex1} or \ref{gentypeex2} until $\ell$ is at least as large as the highest weight, since $K_X = \mathcal{O}_X(1)$.  When $n = 2m+1$ is odd and $n \geq 5$, the highest weight in the example of Proposition \ref{gentypeex1} is $(s_{m+1}-1)/2 > 2^{2^{(n-2)/2}}$.  When $n = 2m$ is even and $n \geq 4$, the highest weight in the example of Proposition \ref{gentypeex2} is $(s_m-1)(2s_m-1)/2 > 2^{2^{(n-2)/2}}$.  Finally, in dimensions $2$ and $3$, the top weights of the examples $X_{10} \subset \mathbb{P}^3(5,2,1,1)$ and $X_{28} \subset \mathbb{P}^4(14,5,4,3,1)$ have top weights $5$ and $14$, which are greater than $2^{2^{(2-2)/2}} = 2$ and $2^{2^{(3-2)/2)}} \doteq 2.66$, respectively.

To prove part (3), let $n$ be an integer at least 5, and let $W$ be a resolution of singularities of the example in Proposition \ref{gentypeex3} of dimension $n$.  Taking a resolution of singularities doesn't alter the plurigenera,
and so $h^0(W,\ell K_W)=0$ for $1\leq \ell < 2^{2^{(n-4)/2}}$ as we want.
For $2\leq n\leq 4$, we just need to show that there is a smooth projective
$n$-fold of general type with geometric genus
$p_g$ equal to zero. In dimension 2, various examples
have been found, the first one due to Godeaux
\cite[section VII.10]{Barth}.
It follows that such varieties exist in every dimension at least 2,
since $p_g(X\times Y)=p_g(X)p_g(Y)$.
\end{proof}

\subsection{Terminal Calabi-Yau Varieties}

Just as in the canonical case, problems 4a) and 4b) have some bound in each
dimension, by Birkar's results.
To find terminal Calabi-Yau varieties of small volume, we can add an additional weight of $1$ to the examples for problem 3a) from one dimension lower.  The resulting examples have volume roughly $1/2^{2^{n/2}}$;
compare our canonical Calabi-Yau examples,
with volume roughly $1/2^{2^n}$.
The proofs of the first two examples are identical to those of the previous section.

\begin{proposition}
\label{tercyex1}
Let $n = 2m+2$ be an even integer at least 4, and let $d = s_{m+1}-1$.  Then a general hypersurface $X$ of degree $2d$ in the weighted projective space $\mathbb{P}^{n+1}((d/s_0)^{(2)},\ldots,(d/s_m)^{(2)},1^{(2)})$ is quasi-smooth and Calabi-Yau with terminal singularities.  When $n \geq 6$, $\mathrm{vol}(\mathcal{O}_X(1)) < 1/2^{2^{n/2}}$.
\end{proposition}

\begin{proposition}
\label{tercyex2}
Let $n = 2m+3$ be an odd integer at least 5, and let $d = (s_m-1)(2s_m-1)$.  Then a general hypersurface $X$ of degree $2d$ in the weighted projective space $Y = \mathbb{P}^{n+1}((d/s_0)^{(2)},\ldots,(d/s_{m-1})^{(2)},2(s_m-1),(s_m-1)^{(2)},1^{(2)})$ is quasi-smooth and Calabi-Yau with terminal singularities.  When $n \geq 7$, $\mathrm{vol}(\mathcal{O}_X(1)) < 1/2^{2^{n/2}}$.
\end{proposition}

Finally, we'll find an example for 4b), a terminal Calabi-Yau variety
with an ample Weil divisor with many vanishing spaces of sections.

\begin{proposition}
\label{tercyex3}
\begin{enumerate}
    \item Let $n = 2m+2$ be an even integer at least $8$, and let $d = (s_m-1)(4s_{m-2}^3 - 6s_{m-2}^2 + 5s_{m-2} -2)$.  Then a general hypersurface $X$ of degree $2d$ in $Y = \mathbb{P}^{n+1}((d/s_0)^{(2)},\ldots,(d/s_{m-1})^{(2)},\newline (s_{m-2}(2s_{m-1}-1))^{(2)},(2(s_{m-2}-1)s_{m-1})^{(2)})$ is quasi-smooth and Calabi-Yau with terminal singularities.  Such an $X$ satisfies $H^0(X, \mathcal{O}_X(\ell)) = 0$ for $1 \leq \ell < M$, where $M = 2(s_{m-2}-1)s_{m-1} > 2^{2^{(n-5)/2}}$.
    \item  Let $n = 2m+3$ be an odd integer at least $9$, and let $d = (s_m-1)(4s_{m-2}^3 - 6s_{m-2}^2 + 5s_{m-2} -2)$.  Then a general hypersurface $X$ of degree $2d$ in
$Y = \mathbb{P}^{n+1}((d/s_0)^{(2)},\ldots,(d/s_{m-3})^{(2)},d/s_{m-2},(d/(2s_{m-2}))^{(2)},(d/s_{m-1})^{(2)},\newline (s_{m-2}(2s_{m-1}-1))^{(2)},(2(s_{m-2}-1)s_{m-1})^{(2)})$
is quasi-smooth and Calabi-Yau with terminal singularities.  Such an $X$ satisfies $H^0(X, \mathcal{O}_X(\ell)) = 0$ for $1 \leq \ell < M$, where $M = 2(s_{m-2}-1)s_{m-1} > 2^{2^{(n-6)/2}}$.
\end{enumerate}

\end{proposition}

\begin{proof}
We'll begin with part (1).  Throughout the proof, we'll abbreviate the last weights as $a := s_{m-2}(2s_{m-1}-1)$ and $b := 2(s_{m-2}-1)s_{m-1}$.  Since all weights are repeated twice, to show that $X$ is well-formed, it suffices to see that $a$ and $b$ are coprime.  This is true by the properties of Sylvester's sequence and the fact that $m \geq 3$, so that $s_{m-2}$ is odd.  All weights divide $2d$ except for $a$, which divides $2d - d/s_{m-1}$.  Indeed, $2d - d/s_{m-1} = d(2s_{m-1} - 1)/s_{m-1}$ and $s_{m-2}$ divides $d$ (in particular, $s_m-1$).  There are two copies of $a$, but there are two copies of $d/s_{m-1}$ as well; so the criterion for quasi-smoothness (Proposition \ref{qsmooth}) is satisfied.

The sum of all but the last four weights is $2d(1-1/(s_m-1)) = 2d - 2(4s_{m-2}^3 - 6s_{m-2}^2 + 5s_{m-2} -2)$, while 
$$a + b = s_{m-2}(2s_{m-1}-1) + 2(s_{m-2}-1)s_{m-1} = 4s_{m-2}^3 - 6s_{m-2}^2 + 5s_{m-2} -2,$$
so $K_X = \mathcal{O}_X$ and the hypersurface is Calabi-Yau.

To show that it is terminal, first note that the sum of each weight taken only once is $d$.  Since each $d/s_j$ divides $d$, the singularity 
$$\frac{1}{d/s_j}(d/s_0,\ldots, \widehat{d/s_j},\ldots,d/s_{m-1}, s_{m-2}(2s_{m-1}-1),2(s_{m-2}-1)s_{m-1})$$
is canonical.  Once we repeat each weight twice, the singularity
is terminal, and so $Y$ is terminal at the coordinate points
of weight $d/s_j$. 

We can't use this argument to get that $Y$ is terminal at the coordinate points of weight $b$ because $b$ divides $2d$, but not $d$.  However, we may use a different subset of weights.  Let $w: = 4s_{m-2}^3 - 6s_{m-2}^2 + 5s_{m-2} -2$ so that $d = (s_m-1)w = (s_{m-2}-1)s_{m-1}s_{m-2}w$ and $d/s_{m-2} = (s_{m-2}-1)s_{m-1}w$.  Both $d$ and $d/s_{m-2}$ are multiples of $b/2$ by an odd number (both $w$ and $s_{m-2}$ are odd since $m \geq 3$), so $d - d/s_{m-2} \equiv 0 \pmod{b}$.  But $d - d/s_{m-2}$ is the sum of each weight repeated once omitting $d/s_{m-2}$.  This subset of weights appears twice in $Y$, so the singularities corresponding to weights $b$ are terminal.

So far, we've proven that $Y$, and hence $X$, are terminal away from the base locus of $\mathcal{O}_Y(2d)$, which is the $\mathbb{P}^1$ corresponding to the two weights $a$.  Regardless of which stratum of the $\mathbb{P}^1$ we consider, applying Proposition \ref{singcriterion} leads to analyzing the singularity $\frac{1}{a}((d/s_0)^{(2)},...,(d/s_{m-2})^{(2)},d/s_{m-1},b^{(2)})$. (Looking at a zero-dimensional stratum would add another weight $a$, but this does not change the singularity type.)

\begin{lemma}
For any integer $m \geq 3$, the quotient singularity
$$\frac{1}{a}((d/s_0)^{(2)},...,(d/s_{m-2})^{(2)},d/s_{m-1},b^{(2)})$$
is terminal and Gorenstein. 
\end{lemma}

This is the hardest case that we encounter.
In this case, there is typically no nonempty proper subset
of the weights whose sum is zero modulo $a$, which would be our most common
approach to proving terminality.

\begin{proof}
The sum of the weights is zero modulo $a$, so the singularity is Gorenstein,
and hence canonical (Lemma \ref{cartier}).  Since $a$ divides $2d - d/s_{m-1}$, it is also possible to write this quotient singularity as 
$$\frac{1}{a}((d/s_0)^{(2)},...,(d/s_{m-2})^{(2)},2d,b^{(2)}).$$
Now, dividing the weight $2d$ into two copies of $d$ gives
$$\frac{1}{a}((d/s_0)^{(2)},...,(d/s_{m-2})^{(2)},d^{(2)},b^{(2)}),$$
which is terminal Gorenstein.  This is because $a$ is odd, so taking a sum of only one copy of each weight gives $0 \bmod a$.  Therefore, we need to show that combining the two weights $d$ does not change the fact that our singularity is terminal.  Suppose by way of contradiction that the original singularity is canonical but not terminal.  Then for some $i$ with $1 \leq i < a$,
$$2\left(ib \bmod a + \sum_{j = 1}^{m-2} id/s_j \bmod a \right) + (2id) \bmod a = a.$$
For the same $i$, the fact that the singularity with $2d$ split is terminal means that
$$2\left( ib \bmod a + \sum_{j=1}^{m-2} id/s_j \bmod a \right) + 2(id \bmod a) = 2a.$$
The last expression must be at least $2a$ because the singularity is terminal and Gorenstein.  It is at most $2a$ because combining the two $d$ weights can lower the sum by no more than $a$.  From now on, fix an $i$ for which the above two
equalities hold.  Taking half of the second expression and setting $K = \sum_{j=1}^{m-2} id/s_j \bmod a + ib \bmod a$, we have that both $2K + (2id) \bmod a = a$ and $K + id \bmod a = a$.
Using that $a$ is odd, the first equation implies that $K<a/2$,
and then the second equation gives that $id \bmod a > a/2$.  In particular, $id \bmod a$ is nonzero.  But $d = d/s_0 + \cdots + d/s_{m-1} + a + b$. Taking the smallest nonnegative residues modulo $a$ of each term, we therefore have
$$id \bmod a \leq \sum_{j = 1}^{m-1} id/s_j \bmod a + i(a+b) \bmod a.$$
(This expresses the fact we've used frequently that dividing a weight into multiple weights with the same sum can only increase the contribution to the Reid-Tai criterion.)
The right-hand side may be rewritten as $ib \bmod a + \sum_{j = 1}^{m-2} id/s_j \bmod a + (id/s_{m-1}) \bmod a = K + (id/s_{m-1}) \bmod a$, so in fact
\begin{equation}
\label{tercyineq}
id \bmod a \leq K + (id/s_{m-1}) \bmod a.
\end{equation}

The two sides of (\ref{tercyineq}) are congruent modulo $a$, so the only way $K < a/2$ and $id \bmod a > a/2$ can hold
is if inequality \ref{tercyineq}
is an equality.  We'll show that this can't happen.

For a start, consider the $id/s_0 = id/2$ term in $K$.  We know that $id \bmod a$ is nonzero (because it is greater than $a/2$), so $id/2 \bmod a$ is as well.  Moreover, $id/2 \bmod a$ is at least $(id \bmod a)/2$ and these can be equal only if $2$ actually divides $id \bmod a$.  Something similar happens with all the other terms: we must have $(id/s_j) \bmod a \geq (id \bmod a)/s_j$ for every $j$ and an analogous statement for the term with $i(a+b) = id/(s_0 \cdots s_{m-1})$, namely $(id/(s_0 \cdots s_{m-1})) \bmod a \geq (id \bmod a)/(s_0 \cdots s_{m-1})$.  In each case, the inequality is only an equality when $(id \bmod a)$ is divisible by the relevant $s_j$ or $s_0 \cdots s_{m-1}$.  Since 
$$id \bmod a = (id \bmod a)(1/s_0 + \cdots + 1/s_{m-1} + 1/(s_0 \cdots s_{m-1})),$$
we must actually be in the equality case for each term.  All this implies that $s_0 \cdots s_{m-1} = s_m-1$ divides $id \bmod a$.  However, $id \bmod a$ is a nonzero integer between $1$ and $a$.  Since $a = s_{m-2}(2s_{m-1}-1) < s_m-1$ when $m \geq 3$, this is a contradiction.
\end{proof}

This concludes the proof of part (1) of the proposition.  The proof of (2) is nearly identical; we've split the weight $d/s_{m-2}$ in two because two copies of $d/s_{m-1}$ are required for quasi-smoothness.  The bound on $M$ comes from the fact that $b = 2(s_{m-2}-1)s_{m-1} > s_{m-2}^3 > (2^{2^{m-3}})^3 > 2^{2^{m-(3/2)}}$.
\end{proof}

The examples from Propositions \ref{tercyex1}, \ref{tercyex2}, and \ref{tercyex3} give the following theorem.

\begin{theorem}
\label{tercytheorem}
\begin{enumerate}
    \item For every integer $n\geq 6$, there is a
 terminal Calabi-Yau $n$-fold $X$ with an ample Weil divisor $A$
such that $\mathrm{vol}(A) < 1/2^{2^{n/2}}$.
    \item For every integer $n\geq 8$, there is
a terminal Calabi-Yau $n$-fold $X$ with an ample Weil divisor $A$
such that $H^0(X,\ell A) = 0$ for $1 \leq \ell < 2^{2^{(n-6)/2}}$.
\end{enumerate}
\end{theorem}

\section{Noether-type inequalities}
\label{sectionfano}

We now deduce from Propositions \ref{gentypeex1} and \ref{gentypeex2}
that the constant $a_n$ in a Noether-type inequality
$\vol(X)\geq a_n p_g(X)-b_n$ (for $n$-folds $X$ of general type)
must be doubly exponentially small as a function of $n$.

Noether's inequality
for surfaces of general type says that $\vol(X)\geq 2p_g-4$, where
the geometric genus $p_g$ means $h^0(X,K_X)$. More generally,
M.~Chen and Z.~Jiang
showed that for every positive integer $n$ there are positive constants
$a_n$ and $b_n$ such that $\vol(X)\geq a_np_g(X)-b_n$ for every
smooth projective $n$-fold $X$ of general type \cite[Corollary 5.1]{CJ}.
Strengthening earlier results, J.~Chen, M.~Chen, and C.~Jiang
recently proved a Noether inequality for 3-folds of general type,
with optimal constants:
we have $\vol(X)\geq (4/3)p_g(X)-10/3$ if $p_g(X)\geq 11$
\cite{CCJ}.

In high dimensions, no explicit constants in Noether's inequality are known,
although they are related to lower bounds for the volume
in lower dimensions. Using that relation, we now show that the constant
$a_n$ must tend rapidly to zero with $n$.
Our argument uses the product of a given variety
with curves of high genus, as suggested
by J.~Chen and C.-J.~Lai \cite[Example 1.4]{CL}.

\begin{theorem}
\label{noether}
For every integer $n\geq 5$,
there is a sequence of smooth complex projective $n$-folds
of general type
with $p_g\to\infty$ and $\vol/p_g< 1/2^{2^{n/2}}$.
\end{theorem}

\begin{proof}
Let $Z$ be the variety of dimension $n-1$ with $p_g(Z)>0$
given by Propositions \ref{gentypeex1} and \ref{gentypeex2}.
For a smooth projective curve $C$ of genus $g\geq 2$,
we have $p_g(Z\times C)=g\, p_g(Z)$ and $\vol(Z\times C)=n(2g-2)\vol(Z)$.
Therefore, taking a sequence of curves $C$ with genera going
to infinity, the $n$-folds $Z\times C$ have $p_g\to\infty$
and $\vol/p_g\to 2n\vol(Z)/p_g(Z)\leq 2n\vol(Z)$. For $n\geq 6$,
this is less than $1/2^{2^{n/2}}$, as we want. For $n=5$,
the 4-fold $Z$ is a resolution of
$X_{20} \subset \mathbb{P}^5(5^{(2)},4,2^{(2)},1)$,
with volume $1/20$, which is not good enough. We can instead
use the 4-fold $X_{64}\subset \mathbb{P}^5(19,16,11,9,7,1)$,
which is terminal with $K_X=\mathcal{O}_X(1)$ \cite{database}. This variety
has $p_g>0$ and volume $4/13167$, which is good enough to imply the theorem
for $n=5$.
\end{proof}

\end{document}